\documentclass[a4paper,oneside,12pt,english]{amsart}

\usepackage[utf8]{inputenc}
\usepackage[T1]{fontenc}
\usepackage{babel}
\usepackage{geometry}

\usepackage{amsmath,amsthm,amssymb}
\usepackage{mathtools}
\usepackage{esint}

\usepackage{enumitem}

\usepackage{lmodern}

\usepackage{microtype}

\usepackage{graphicx}
\usepackage{float}

\usepackage{longtable}

\usepackage{hyperref}
\hypersetup{unicode=true,colorlinks=true,citecolor=black}

\usepackage{dsfont}  
\usepackage{mathrsfs} 

\usepackage{accents}

\theoremstyle{plain}
\newtheorem{theorem}{Theorem}[section]
\newtheorem{lemma}[theorem]{Lemma}
\newtheorem{corollary}[theorem]{Corollary}

\theoremstyle{definition}
\newtheorem{definition}[theorem]{Definition}

\theoremstyle{remark}
\newtheorem{remark}[theorem]{Remark}

\newcommand{\R}{\mathbb{R}}
\newcommand{\CC}{\mathbb{C}}
\newcommand{\Sp}{\mathbb{S}}
\newcommand{\BMO}[0]{\operatorname{BMO}}
\newcommand{\VMO}[0]{\operatorname{VMO}}
\newcommand{\Lip}[0]{\operatorname{Lip}}
\newcommand{\supp}[0]{\operatorname{supp}}
\newcommand{\eps}[0]{\varepsilon}

\newcommand{\abs}[1]{\left| #1 \right|}
\newcommand{\dif}{\mathop{}\!\mathrm{d}}
\newcommand{\ave}[1]{\langle #1\rangle}
\newcommand{\loc}[0]{\operatorname{loc}}
\newcommand{\wt}[1]{\widetilde{#1}}

\begin{document}
\title[Compactness of commutators of rough singular integrals]{Compactness of commutators of rough singular integrals}

\begin{abstract}
    We study the two-weighted off-diagonal compactness of commutators of rough singular integral operators $T_\Omega$ that are associated with a kernel $\Omega\in L^q(\Sp^{d-1})$. We establish a characterisation of compactness of the commutator $[b,T_\Omega]$ in terms of the function $b$ belonging to a suitable space of functions with vanishing mean oscillation. Our results expand upon the previous compactness characterisations for Calder\'on-Zygmund operators. Additionally, we prove a matrix-weighted compactness result for $[b,T_\Omega]$ by applying the so-called matrix-weighted Kolmogorov-Riesz theorem.
\end{abstract}

\author{Aapo Laukkarinen}
\address{Aapo Laukkarinen, Department of Mathematics and Systems Analysis, Aalto University, P.O. Box 11100, FI-00076 Aalto, Finland}
\email{aapo.laukkarinen@aalto.fi}
\author{Jaakko Sinko}
\address{Jaakko Sinko, Department of Mathematics and Statistics, University of Helsinki, P.O. Box 68, FI-00014 Helsinki, Finland}
\email{jaakko.sinko@helsinki.fi}

\maketitle

\section{Introduction}
The commutator $[b,T]$ is defined by 
\[
[b,T]f = bTf - T(bf),
\]
where $T\colon L^2(\R^d)\to L^2(\R^d)$ is a singular integral operator and $b$ is a locally integrable complex-valued function on $\R^d$. The characterisation of the compactness of $[b,T]\colon L^u(\mu)\to L^v(\lambda)$, where $\mu\in A_u$ and $\lambda\in A_v$, via the membership of $b$ in some function space has steadily seen more results in recent years. We mention here the $1<u=v<\infty$ case of \cite{LL2022}, the $1<u<v<\infty$ case of \cite{HOS2023} and the unweighted $1<v<u<\infty$ case of \cite{HLTY2023}. When it comes to the first two papers, they give their characterisations for $T$ that are associated to what might be called standard kernels or Calder\'on-Zygmund kernels: they are kernels that satisfy standard size and smoothness estimates.

For rough homogeneous kernels 
\[
K_\Omega(x,y)=\Omega\left(\frac{x-y}{|x-y|}\right)\frac{1}{|x-y|^d},
\]
where $\Omega$ has vanishing integral over the unit sphere $\Sp^{d-1}$, the corresponding singular integral operator $T_\Omega$ is formally defined as 
\[
T_\Omega f(x) = \lim_{\eps \to 0} \int_{y\in \R^d\colon |x-y|>\eps} K_\Omega(x,y) f(y) \dif y.
\]
The characterisation result of the paper \cite{HLTY2023} handles both standard kernels and homogeneous kernels. The word ``rough'' indicates that $\Omega$ is not assumed to have continuity, and it only satisfies some integrability condition.

For rough kernels, the compactness characterisation has more parts missing than the corresponding boundedness characterisation. Let us briefly consider some recent advances for the characterisation of boundedness of $[b,T_\Omega]$. Although the present paper is more about the compactness, the study of boundedness shares some methods with the study of compactness. Both a proof of the suitable upper bound for $\|[b,T_\Omega]\|_{L^u(\mu)\to L^v(\lambda)}$ and the suitable lower bound have garnered interest, often in separate papers. When it comes to the upper bound, K.\ Li (\cite{KLi2022}) sketched the proof of 
\begin{equation}\label{eq:kli}
\|[b,T_\Omega]\|_{L^u(\mu)\to L^u(\lambda)} \lesssim \|b\|_{\BMO_\nu}.
\end{equation}
The quantity $\|b\|_{\BMO_\nu}$ is finite if and only if $b$ has bounded mean oscillation with respect to the weight $\nu=\mu^{1/u}\lambda^{-1/u}$. Later, in \cite{LLO2024}, the authors revisited \eqref{eq:kli} and extended it to all $1<u,v<\infty$, with $\BMO_\nu$ replaced by a suitable function space when $u\neq v$. A common assumption to both \cite{KLi2022} and \cite{LLO2024} is that $\Omega \in L^\infty(\Sp^{d-1})$. 

The scheme in  \cite{LLO2024} is that the boundedness of $\Omega$ allows one to control a certain grand maximal function related to $T_\Omega$, and this implies a sparse domination result for the commutators of $T_\Omega$, which is then used as the main ingredient in the proof of the off-diagonal two weighted boundedness. The grand maximal control of $T_\Omega$ is not known  when $\Omega$ is unbounded, and it seems to be a difficult problem. However, in \cite{Lau2023} it was shown that the vector-valued technique called convex body domination also implies the aforementioned sparse domination of commutators. Subsequently, in a recent preprint \cite{Lau2024} convex body domination was proven for singular integrals $T_\Omega$ with  $\Omega\in L^{q,1}\log L(\Sp^{d-1})$ (see Definition \ref{definition:lorentzorlich}), $1<q<\infty$, which then led to the $u\leq v$ upper bound
\begin{equation}\label{eq:al}
\|[b,T_\Omega]\|_{L^u(\mu)\to L^v(\lambda)} \lesssim \|b\|_{\BMO_\nu^\alpha},
\end{equation}
where $\mu\in A_u\cap RH_{\left(\frac{q}{u}\right)'}$, $\lambda\in A_v\cap RH_{\left(\frac{q}{v}\right)'}$, $\alpha/d=1/u-1/v$ and $\nu^{1+\alpha/d}=\mu^{1/u}\lambda^{-1/v}$. Also, one assumes that $1<u,v<q$ and thus the integrability parameter $q$ of $\Omega$ limits the allowed exponent range for $u$ and $v$.

The main result of this paper is a characterisation for the $u\leq v$ compactness of $[b,T_\Omega]$; this result can be seen as a sequel to \cite{Lau2024}. In particular, we show that if $\lambda,\mu$ and $\alpha$ are as in \eqref{eq:al} and  $\Omega\in L^q(\Sp^{d-1})$ is not zero, then \begin{equation}[b,T_\Omega]\colon L^u(\mu)\to L^v(\lambda) \text{ is compact if and only if } b\in \VMO^\alpha_\nu(\R^d).\end{equation} 
This is exactly what was proven in \cite{HOS2023} for standard kernels. For rough $\Omega$ the missing piece in \cite{HOS2023} was a suitable sparse domination for $[b,T_\Omega]$, and thus the direction ``$b\in \VMO_\nu^\alpha$ $\Rightarrow$ $[b,T_\Omega]\colon L^u(\mu)\to L^v(\lambda)$ compact'' could not be stated for rough kernels in that paper. The idea of the proof of this direction is to approximate $[b,T_\Omega]$ with compact commutators $[b,T_\eps]$ and use \eqref{eq:al} to prove that the error commutator term has operator norm that converges to zero. An $L^q(\Sp^{d-1})$ version of the inequality \eqref{eq:al} and a density argument reduces the problem to showing compactness of commutators with Lipschitz functions $\Omega$, which follows from the standard kernel result of \cite{HOS2023}.
For the other direction ``$[b,T_\Omega]\colon L^u(\mu)\to L^v(\lambda)$ compact $\Rightarrow$ $b\in \VMO_\nu^\alpha$'', the methods of \cite{HOS2023} work directly. 

For the sake of comparison with an existing result, we say a few words regarding the quite general sparse domination result of \cite[Theorem 3.2]{LLO2024}. The authors of \cite{LLO2024} successfully applied their sparse domination to prove boundedness of commutators $[b,T_\Omega] \colon L^u(\mu) \to L^v(\lambda)$ under the assumption $\Omega \in L^\infty(\Sp^{d-1})$. However, since $\Lip(\Sp^{d-1})$ is not dense in $L^\infty(\Sp^{d-1})$, one can not pair their sparse domination with the same density argument that we use. This is the benefit of the assumption $\Omega \in L^q(\Sp^{d-1})$. It allows us to reduce our proof associated to rough homogeneous kernels to the analogous proof associated to standard kernels via a density argument.

We will also study compactness of commutators in a matrix-weighted space. The main tool we will use to tackle matrix-weighted compactness is the matrix-weighted Kolmogorov-Riesz theorem of \cite{LYZ2023} (see Theorem \ref{thm:kolmogorovriesz}). In particular, we will show that for $b\in\VMO$, $1<q<\infty$, $q'<p<\infty$, $\Omega\in L^{q+\eps}(\Sp^{d-1})$, $\eps>0$, the commutator $[b,T_\Omega]$ is $L^p(W)\to L^p(W)$ compact.

The rest of the paper is organised as follows. In Section \ref{sec:prelims} we gather the relevant definitions and some preliminary results that will be used later in the paper. In Section \ref{sec:sufficiencyforcompactness} we will prove the main result of this paper, that is, $b\in \VMO^\alpha_\nu$ is a sufficient condition for Bloom-type compactness. Then in Section \ref{sec:necessityforcompactness} we will see that the methods of \cite{HOS2023} can be used to prove that $b\in \VMO^\alpha_\nu$ is also a necessary condition. Lastly, in Section \ref{sec:matrixweightedcompactness} we will consider matrix-weighted compactness.

\subsection{Notation}Some of the used notation is summarized in the following table:

\begin{longtable}{c p{0.7\textwidth}}
$\mu$ & $A_u$ weight in $\R^d$.  \\
$\lambda$ & $A_v$ weight in $\R^d$.  \\
$u'$ & Conjugate exponent of $u\in(1,\infty)$: $\frac{1}{u}+\frac{1}{u'}=1$.\\
$\alpha$ & Exponent defined by $\frac{\alpha}{d}=\frac{1}{u}-\frac{1}{v}$.\\
$\nu$ & Bloom weight defined by
$
\nu^{1+\alpha/d} = \mu^{1/u}\lambda^{-1/v}.
$ \\

\\

$1_E$ & Indicator function of the set $E\subset \R^d$. \\
$\ave{f}_E$ & Average: $\ave{f}_E=\frac{1}{|E|}\int_E f(x) \dif x$. \\
$w(E)$ & Weighted measure $w(E)=\int_E w(x) \dif x$. \\

\\

$Q$ & A cube in $\R^d$ with sides parallel to the coordinate axes. \\
$\ell(Q)$ & The side length of a cube $Q$. \\
$\Sp^{d-1}$ & The unit sphere $\{x \in \R^d : |x|=1\}$. \\ 
$\|b\|_{\BMO_w^{\alpha}}$ & $\|b\|_{\BMO_w^{\alpha}} \! = \! \|b\|_{\BMO_w^{\alpha}(\R^d)} \! = \! \sup_{Q} \frac{1}{w(Q)^{1+\alpha/d}}\int_{Q}|b(x)-\ave{b}_{Q}| \dif x$. \\
$\|f\|_{L^p(w)}$ & $\|f\|_{L^p(w)}=\left(\int_{\R^d} |f(x)|^pw(x) \dif x \right)^{1/p}$. \\

\\

$A \lesssim_P B$ & Inequality with an implicit constant $C_P>0$ that depends at most on parameters $P$, i.e.\ $A \leq C_P B$. \\
$A \eqsim_P B$ & Indicates that both $A \lesssim_P B$ and $B \lesssim_P A$ hold.
\end{longtable}

\section{Preliminaries}\label{sec:prelims}

Throughout, unless specified otherwise, all function spaces ($\BMO_w^\alpha(\R^d)$, \\ $\VMO_w^\alpha(\R^d)$, $L^p(w)$, $\Lip(\Sp^{d-1})$, $L^{q,1}\log L(\Sp^{d-1})$, $L^q(\Sp^{d-1})$, $\ldots$) are allowed to contain complex-valued measurable functions.

By a positive weight we mean a locally integrable function $w$ that satisfies $w>0$ almost everywhere.

\begin{definition}
Let $w$ be a positive weight and $\alpha \in \R$. We define the space $\BMO_w^\alpha(\R^d)$ to consist of the locally integrable functions $b$ that satisfy 
\begin{align*}
\sup_{Q} \frac{1}{w(Q)^{1+\alpha/d}}\int_Q\abs{b-\ave{b}_Q} \dif x <\infty,
\end{align*}
where the symbol $Q$ denotes any cube in $\R^d$.
\end{definition}

\begin{definition}
Let $w$ be a positive weight and $\alpha \in \R$. We define the space $\VMO_w^\alpha(\R^d)$ to consist of the locally integrable functions $b$ that satisfy 
\begin{align*}
\lim_{s\to 0}\sup_{Q\colon l(Q)\leq s} \frac{1}{w(Q)^{1+\alpha/d}}\int_Q\abs{b-\ave{b}_Q} \dif x = 0,
\end{align*}
\begin{align*}
\lim_{s\to \infty}\sup_{Q\colon l(Q)\geq s} \frac{1}{w(Q)^{1+\alpha/d}}\int_Q\abs{b-\ave{b}_Q} \dif x = 0,
\end{align*}
\begin{align*}
\lim_{s\to \infty}\sup_{Q\colon \mathrm{dist}(Q,0)\geq s} \frac{1}{w(Q)^{1+\alpha/d}}\int_Q\abs{b-\ave{b}_Q} \dif x = 0,
\end{align*}
where the symbol $Q$ denotes any cube in $\R^d$.
\end{definition}

Often in the definition of a $\VMO$ function, one assumes a priori that the function is a $\BMO$ function. We show that in our setting, even if one assumes only local integrability, it will a posteriori hold that the function is a $\BMO$ function.

\begin{lemma}\label{lemma:vmoinsidebmo}
     $\VMO_w^\alpha(\R^d) \subset \BMO_w^\alpha(\R^d)$ for all positive weights $w$ and $\alpha \in \R$.
\end{lemma}
\begin{proof}
For any cube $Q \subset \R^d$, we denote 
\[
\mathcal{O}_{w}^{\alpha}(b;Q) := \frac{1}{w(Q)^{1+\alpha/d}}\int_Q\abs{b-\ave{b}_Q} \dif x
\]

First choose a large $s>0$ so that if $l(Q)\leq 1/s$, $l(Q)\geq s$ or $\mathrm{dist}(Q,0)\geq s$, we have $\mathcal{O}_{w}^{\alpha}(b;Q)\leq 1$. Let us turn our attention to the remaining cubes $Q$ that satisfy $1/s<l(Q)<s$ and $\mathrm{dist}(Q,0)< s$. Because $l(Q)<s$ and $\mathrm{dist}(Q,0)< s$, all such cubes $Q$ are contained in some cube $Q_0$ centred at origin with side length $l(Q_0)$. Therefore 
\[
\int_{Q}|b-\ave{b}_Q|\leq 2\int_{Q}|b-\ave{b}_{Q_0}|\leq 2\int_{Q_0}|b-\ave{b}_{Q_0}|= 2w(Q_0)^{1+\alpha/d}\mathcal{O}_{w}^{\alpha}(b;Q_0)
\]
for the remaining cubes $Q$. Let us choose a positive integer $k=k(s)$ so that $2^{-k}l(Q_0)\leq\frac{1}{2s}$. Let $\{P_1,\ldots, P_{2^{kd}}\}$ consist of the $2^{kd}$ $k$th level dyadic subcubes $P_j$ of $Q_0$. We denote $m:=\min\{w(P_1),\ldots, w(P_{2^{kd}})\}>0$. Let $Q$ be one of the remaining cubes. Then there exists an index $j$ such that the centre point of $Q$ is in $P_j$. Since $l(Q)>1/s$ and $2^{-k}l(Q_0)\leq \frac{1}{2s}$, it holds that $P_j\subset Q$. Therefore $m\leq w(P_j)\leq w(Q)$ and 
\[
\frac{1}{w(Q)^{1+\alpha/d}}\leq\max\big\{\frac{1}{m^{1+\alpha/d}},\frac{1}{w(Q_0)^{1+\alpha/d}}\big\}. 
\]
Note that whether or not $1+\alpha/d\geq0$ affects the proof of the above estimate and combining it with the earlier estimate, we get 
\begin{align*}
\mathcal{O}_{w}^{\alpha}(b;Q)\leq 2\max\big\{\frac{w(Q_0)^{1+\alpha/d}}{m^{1+\alpha/d}},1\big\}\mathcal{O}_{w}^{\alpha}(b;Q_0).
\end{align*}
for the remaining cubes $Q$. Thus 
\[
\|b\|_{\mathrm{BMO}_{w}^{\alpha}(\R^d)}\leq \max\{1,2\max\big\{\frac{w(Q_0)^{1+\alpha/d}}{m^{1+\alpha/d}},1\big\}\mathcal{O}_{w}^{\alpha}(b;Q_0)\}<\infty.
\]
\end{proof}

\begin{remark}
    For our purposes, we study the spaces $\BMO_w^\alpha(\R^d)$ and $\VMO_w^\alpha(\R^d)$ only for $\alpha \geq 0$.
\end{remark}

\begin{definition}[Muckenhoupt classes of weights]
    Let $1<u<\infty$ and let $w$ be a positive weight. Then we say that $w\in A_u$ if 
    \[
    \sup_Q\ave{w}_Q\ave{w^{1-u'}}_Q^{u-1} < \infty,
    \]
    where the symbol $Q$ denotes any cube in $\R^d$.
\end{definition}

\begin{definition}[Reverse H\"older classes of weights]
    Let $1<u<\infty$ and let $w$ be a positive weight. Then we say that $w\in RH_u$ if
    \[
    \sup_Q \ave{w^u}_Q^\frac{1}{u}\ave{w}_Q^{-1} <  \infty,
    \]
    where the symbol $Q$ denotes any cube in $\R^d$.
\end{definition}

\begin{definition}
    Suppose $\Omega \colon \Sp^{d-1} \to \CC$. We say that $\Omega \in \Lip(\Sp^{d-1})$, if there exists $M>0$ so that 
    \[
    |\Omega(x)-\Omega(y)| \leq M|x-y| \quad \forall x,y\in \Sp^{d-1}. 
    \]
\end{definition}

\begin{definition}\label{tomega}
    Let $1 < q \leq \infty$. Suppose that $\Omega \in L^q(\Sp^{d-1})$. Suppose also that the integral of $\Omega$ over $\Sp^{d-1}$ is zero, or in other words, that $\Omega$ has zero average. We formally define
    \[
    T_\Omega f(x) := \lim_{\eps \to 0} \int_{|y|>\eps} \frac{\Omega(y/|y|)}{|y|^d}f(x-y) \dif y.
    \]
    We call the kernel $K_\Omega$ defined by $K_\Omega(x,y):=\frac{\Omega((x-y)/|x-y|)}{|x-y|^d}$ a \emph{rough homogeneous} kernel.
\end{definition}

We record the well-known boundedness of $T_\Omega$. For the proof, see \cite[Theorem 4.12]{Duo2001}.

\begin{lemma}\label{lemma:lqomegameansbounded}
     Let $1<q\leq\infty$ and suppose that $\Omega \in L^q(\Sp^{d-1})$ has zero average over $\Sp^{d-1}$. Then the linear operator $T_\Omega$ defined by
     \[
     T_\Omega(f)=\lim_{\eps\to 0} \int_{|y|>\eps} \frac{\Omega(y/|y|)}{|y|^d}f(\cdot-y) \dif y
     \]
    is bounded on $L^p(\R^d)$ for every $1<p<\infty$.
\end{lemma}

We will consider a kernel that is connected to the following Orlicz-Lorentz type space.

\begin{definition}\label{definition:lorentzorlich}
    Let $1< q< \infty$. Following \cite{Lau2024}, we define
    \[
        \|\Omega\|_{L^{q,1}\log L(\Sp ^{d-1})} \coloneqq \inf\{\lambda>0\,\colon\,[\Omega/\lambda]_{L^{q,1}\log L(\Sp ^{d-1})}\leq1\},
    \]
    where
    \[
    [\Omega]_{L^{q,1}\log L(\Sp ^{d-1})} \coloneqq q\int_0^\infty \log(e+t)|\{\theta\in \Sp^{d-1} : |\Omega(\theta)|>t\}|^\frac{1}{q} \dif t.
    \]
    We say that  $\Omega\in L^{q,1}\log L(\Sp^{d-1})$, if $\|\Omega\|_{L^{q,1}\log L(\Sp ^{d-1})}<\infty$.
\end{definition}

One key property of $L^{q,1}\log L(\Sp^{d-1})$ is that it is between $L^q(\Sp^{d-1})$ and $L^{q+\eps}(\Sp^{d-1})$ for any $\eps>0$. This might be well-known to experts, but we could not find a proof and record it here for completeness. The following definition will help with the proof of this fact. 
\begin{definition}
    The decreasing rearrangement $f^*$ of a measurable function $f$ is defined by
    \[
        f^*(t)=\inf\{\tau>0\,\colon\,|\{|f|>\tau\}|\leq t\}.
    \]
\end{definition}
We also find the following special case of Lemma 6.1 from \cite{BR1980} very helpful.
\begin{lemma}\label{lemma:decfuncineq}
    Suppose $0<a,b<\infty$. Let $\phi$ be a non-negative decreasing function on $(0,\infty)$. Then for each $t\in(0,\infty)$, we have
    \[
        \sup_{0<s\leq t}s^b\phi(s)\leq \left(ab\int_0^t \left[s^b\phi(s)\right]^a \frac{\dif s}{s}\right)^\frac{1}{a}.
    \]
\end{lemma}
Note that in \cite{BR1980} the above result was formulated only for non-negative decreasing functions $\phi$ on $(0,1)$ and $t\in(0,1)$, but an inspection of the proof shows that the above is also true. 
\begin{lemma}\label{lemma:lorentzorlicztolebesgue}
Suppose that $1<q<\infty$. Then for any $\eps>0$ we have \[L^{q+\eps}(\Sp^{d-1})\subset L^{q,1}\log L(\Sp^{d-1})\subset L^{q}(\Sp^{d-1}).\] Furthermore, there is a number $C>0$ that depends only on $d,q$ and $\varepsilon$ so that 
\[
\|f\|_{L^q(\Sp^{d-1})}\leq\|f\|_{L^{q,1}\log L(\Sp^{d-1})} \leq C  
\|f\|_{L^{q+\eps}(\Sp^{d-1})}.
\]
\end{lemma}
\begin{proof}
    The first inequality is quite straightforward. We have 
    \[
        \|f\|_{L^q(\Sp^{d-1})}\leq \|f\|_{L^{q,1}(\Sp^{d-1})}\leq [f]_{L^{q,1}\log L(\Sp^{d-1})}.
    \]
    Scaling this inequality with $\lambda<\|f\|_{L^q(\Sp^{d-1})}$ yields
    \[
        [f/\lambda]_{L^{q,1}\log L(\Sp^{d-1})}\geq \frac{\|f\|_{L^q(\Sp^{d-1})}}{\lambda}>1,
    \]
    which by Definition \ref{definition:lorentzorlich} implies $\lambda<\|f\|_{L^{q,1}\log L(\Sp^{d-1})}$. Letting $\lambda\to \|f\|_{L^q(\Sp^{d-1})}$ we get the first inequality
    \[
        \|f\|_{L^q(\Sp^{d-1})}\leq \|f\|_{L^{q,1}\log L(\Sp^{d-1})}.
    \]

    To prove the second inequality we will first estimate $[f]_{L^{q,1}\log L(\Sp^{d-1})}$.
    By calculating the integral of the function $\log(e+t) r^{\frac{1}{q}-1}$
    over the subset $\{|\{|f|>t\}|>r\}=\{f^*(r)>t\}$ of $(0,\infty)\times(0,\infty)$ in two different ways, we get that
    \begin{align*}
          [f]_{L^{q,1}\log L(\Sp^{d-1})}&= q\int_0^\infty \log(e+t)|\{\theta\in \Sp^{d-1} \colon |f(\theta)|>t\}|^\frac{1}{q} \dif t\\&= \int_0^\infty  t^{\frac{1}{q}}(e+f^*(t))(\log(e+f^*(t))-1) \frac{\dif t}{t}.
    \end{align*}
    Note that if $f^*(t)$ vanishes, then the whole integrand vanishes. In particular, this happens when $t\geq |\Sp^{d-1}|\eqqcolon C_d$.  Fix $\delta(\eps)\coloneqq\delta>0$ to be chosen later and denote $r\coloneqq q+\delta$. Then we estimate
    \begin{align*}
        \int_0^{C_d}  t^{\frac{1}{q}}&(e+f^*(t))(\log(e+f^*(t))-1) \frac{\dif t}{t}\\&\leq \sup_{0<t\leq C_d}t^\frac{1}{r}(e+f^*(t))(\log(e+f^*(t))-1)\int_0^{C_d} t^{\frac{1}{q}-\frac{1}{r}-1}\dif t.
    \end{align*}
    Since $r>q$ the integral satisfies
    \[
        \int_0^{C_d} t^{\frac{1}{q}-\frac{1}{r}-1}\dif t=C_d^{\frac{1}{q}-\frac{1}{r}}\frac{qr}{r-q}\eqsim_{d,q,\eps}1.
    \]
    An application of Lemma \ref{lemma:decfuncineq} yields that 
    \begin{align*}
        \sup_{0<t\leq C_d}t^\frac{1}{r}&(e+f^*(t))(\log(e+f^*(t))-1)\\&\leq \left(\int_0^{C_d}\left[t^\frac{1}{r}(e+f^*(t))(\log(e+f^*(t))-1)\right]^r\frac{\dif t}{t}\right)^\frac{1}{r}\\&=  \left(\int_0^{C_d}\left[(e+f^*(t))(\log(e+f^*(t))-1)\right]^r\dif t\right)^\frac{1}{r}.
    \end{align*}
    By Minkowski's inequality the last expression is bounded by 
    \[
        e\left(\int_0^{C_d}\left[\log(e+f^*(t))-1\right]^r\dif t\right)^\frac{1}{r}+\left(\int_0^{C_d}\left[f^*(t)(\log(e+f^*(t))-1)\right]^r\dif t\right)^\frac{1}{r}
    \]
    For any $0<a\leq1$ and $x\geq0$ there holds $\log(e+x)-1\leq \mathcal C_{a} x^{a}$, where 
    \[
        \mathcal C_{a}\coloneqq \sup_{x>0}\frac{\log(e+x)-1}{x^a}.
    \]
    This with $a=1$ and $x=f^*(t)$ lets us estimate 
    \[
        \left(\int_0^{C_d}\left[\log(e+f^*(t))-1\right]^r\dif t\right)^\frac{1}{r}\lesssim \|f\|_{L^{q+\delta}(\Sp^{d-1})}.
    \]
    For the second integral we use the same estimate with $a=\frac{\delta}{r}$ and $x=f^*(t)$, which yields
    \begin{align*}
        \left(\int_0^{C_d}\left[f^*(t)(\log(e+f^*(t))-1)\right]^r\dif t\right)^\frac{1}{r}&\lesssim_{q,\eps} \left(\int_0^{C_d}\left[f^*(t)\right]^{r+\delta}\dif t\right)^\frac{1}{r}\\&=\|f\|_{L^{r+\delta}(\Sp^{d-1})}^{1+\frac{\delta}{r}}=\|f\|_{L^{q+2\delta}(\Sp^{d-1})}^{1+\frac{\delta}{q+\delta}}.
    \end{align*}
    Choosing $\delta=\frac{\eps}{2}$ and combining all the estimates we have now shown that there exists a constant $C\coloneqq C(d,q,\eps)$ such that
    \[
        [f]_{L^{q,1}\log L(\Sp^{d-1})}\leq C\left(\|f\|_{L^{q+\eps}(\Sp^{d-1})}+\|f\|_{L^{q+\eps}(\Sp^{d-1})}^{1+\frac{\eps}{2q+\eps}}\right).
    \]
    We scale this inequality with $\lambda=2C\|f\|_{L^{q+\eps}(\Sp^{d-1})}$ to get
    \begin{align*}
        [f/\lambda]_{L^{q,1}\log L(\Sp^{d-1})}&\leq C\left(\left\|f/\lambda\right\|_{L^{q+\eps}(\Sp^{d-1})}+\left\|f/\lambda\right\|_{L^{q+\eps}(\Sp^{d-1})}^{1+\frac{\eps}{2q+\eps}}\right)\\&=C\left(\frac{1}{2C}+\frac{1}{(2C)^{1+\frac{\eps}{2q+\eps}}}\right)\leq 1.
    \end{align*}
    By Definition \ref{definition:lorentzorlich} the above implies 
    \[
        \|f\|_{L^{q,1}\log L(\Sp^{d-1})}\leq \lambda = 2C \|f\|_{L^{q+\eps}(\Sp^{d-1})}.
    \]
    This concludes the proof of the lemma.
\end{proof}

\section{Sufficiency for Bloom-type compactness}\label{sec:sufficiencyforcompactness}

In this section we will prove the main result, that is, the Bloom-weighted compactness of commutators of $T_\Omega$ with $\Omega\in L^q(\Sp^{d-1})$.
The first step towards this result is the following Theorem \ref{theorem:bloomBddforLOkernels}, which says that Bloom-type boundedness for the commutator of rough singular integrals holds with $\Omega\in L^{q,1}\log L(\Sp^{d-1})$. As a corollary, we obtain that Bloom-type boundedness holds also with $\Omega\in L^q(\Sp^{d-1})$. The proof will use Lemma \ref{lemma:lorentzorlicztolebesgue} and the self-improvement properties of the reverse H\"older classes.

\begin{theorem}[{\cite[Corollary 7.2]{Lau2024}}]\label{theorem:bloomBddforLOkernels}
    Suppose that $\Omega\in L^{q,1}\log L(\Sp^{d-1})$ for some $1<q<\infty$ with zero average over $\Sp^{d-1}$. Let also $1<u\leq v<q$, $\mu\in A_u\cap RH_{\left(\frac{q}{u}\right)'}$ and $\lambda\in A_v\cap RH_{\left(\frac{q}{v}\right)'}$, and define the Bloom weight $\nu^{1+\frac{\alpha}{d}}\coloneqq \mu^\frac{1}{u}\lambda^{-\frac{1}{v}}$, where $\alpha\coloneqq d\left(\frac{1}{u}-\frac{1}{v}\right)$. Then we have
    \begin{align*}
        \|[b,T_\Omega]\|_{L^u(\mu)\to L^v(\lambda)}\lesssim_{d,q,u,v,\mu,\lambda} \|\Omega\|_{L^{q,1}\log L(\Sp^{d-1})}\, 
        \|b\|_{\BMO^{\alpha}_\nu}.
    \end{align*}
\end{theorem}

\begin{corollary}\label{corollary:bloomBddforLqkernels}
    Theorem \ref{theorem:bloomBddforLOkernels} holds with $L^q(\Sp^{d-1})$ in place of $L^{q,1}\log L(\Sp^{d-1})$. The conclusion is then naturally replaced by 
    \begin{align*}
        \|[b,T_\Omega]\|_{L^u(\mu)\to L^v(\lambda)}\lesssim_{d,q,u,v,\mu,\lambda} \|\Omega\|_{L^q(\Sp^{d-1})}\, 
        \|b\|_{\BMO^{\alpha}_\nu}.
    \end{align*}
\end{corollary}
\begin{proof}
     We will begin by showing that for any $\mu\in RH_{(\frac{q}{u})'}$ and any $\lambda \in RH_{(\frac{q}{v})'}$ there exists a $\delta\in(0,q-v)$ such that $\mu\in RH_{(\frac{q-\delta}{u})'}$ and $\lambda \in RH_{(\frac{q-\delta}{v})'}$.

    We choose $\delta_\lambda=\min\{\frac{q-v}{2},\frac{(q-v)^2}{2v}\eta\}$, where $\eta$ is the number in Gehring's Lemma (See \cite[Lemma 3]{Geh1973}) for $\lambda$. A simple calculation shows that 
    \[
        \left(\frac{q-\delta_\lambda}{v}\right)'
        =\left(\frac{q}{v}\right)'+\delta_\lambda\frac{v}{(q-v-\delta_\lambda)(q-v)}\leq \left(\frac{q}{v}\right)'+\delta_\lambda\frac{2v}{(q-v)^2}\leq \left(\frac{q}{v}\right)'+\eta ,
    \]
    and hence by Gehring's Lemma  we have $\lambda\in RH_{(\frac{q-\delta_\lambda}{v})'}$.
    A similar argument for $\mu$ with some $\delta_\mu\in(0,q-u)$ gives $\mu\in RH_{(\frac{q-\delta_\mu}{u})'}$. Thus 
    $\delta=\min\{\delta_\mu,\delta_\lambda\}$ has all the wanted properties.

    We note also that $\delta$ depends only on $d,q,u,v,\mu$ and $\lambda$, and it satisfies $q-\delta>v$. Thus Theorem \ref{theorem:bloomBddforLOkernels} and Lemma \ref{lemma:lorentzorlicztolebesgue} yields that 
    \begin{align*}
        \|[b,T_\Omega]\|_{L^{u}(\mu)\to L^{v}(\lambda)}&\lesssim_{d,q,u,v,\mu,\lambda} \|\Omega\|_{L^{q-\delta,1}\log L(\Sp^{d-1})}\, 
        \|b\|_{\BMO^{\alpha}_\nu}\\&\lesssim_{d,q,u,v,\mu,\lambda} \|\Omega\|_{L^{q}(\Sp^{d-1})}\, 
        \|b\|_{\BMO^{\alpha}_\nu},
    \end{align*}
    which concludes the proof.
    
\end{proof}

We aim to use a result of \cite{HOS2023} for Calder\'{o}n-Zygmund operators as a black box. In order to do this, we note that if $\Omega$ is smooth enough, it actually induces a standard kernel $K_\Omega$ of a Calder\'{o}n-Zygmund operator. This is recorded as the following lemma. We believe that its proof is definitely known. However, we could not find a direct reference. For the interested reader, we wrote the rest of the details of its proof in Appendix \ref{appendixa}.

\begin{lemma}\label{lemma:smoothomegagivesczo}
Suppose that $\Omega \in \Lip(\Sp^{d-1})$ has zero average over $\Sp^{d-1}$. Then $T_\Omega$ is a Calder\'{o}n-Zygmund operator in the sense of \cite[Definition 1.2]{HOS2023}.
\end{lemma}
\begin{proof}
    Denote $K_\Omega(x,y):=\Omega((x-y)/|x-y|)|x-y|^{-d}$ for all $x,y\in \R^d$ such that $x\neq y$. Suppose then that $f$ is a Schwartz function and that $x$ is outside the support of $f$. When $\eps$ is less than the distance $d$ between $x$ and the support of $f$, we have 
    \[
    \int_{|y|>\eps} \frac{\Omega(y/|y|)}{|y|^d}f(x-y) \dif y = \int_{|y-x|>\eps} K_\Omega (x,y)f(y) \dif y = \int_{\R^d} K_\Omega(x,y) f(y) \dif y.
    \]
    This guarantees the desired off-support representation for $T_\Omega$ with the kernel $K_\Omega$. Moreover, by Lemma \ref{lemma:lqomegameansbounded}, the associated singular integral operator $T_\Omega$ is bounded on $L^2$. The rest of the required properties are checked in Appendix \ref{appendixa}.
\end{proof}

\begin{lemma}\label{lemma:ifsmoothusehosandvmoforcompactness}
    Suppose that $\Omega \in \Lip(\Sp^{d-1})$ has zero average over $\Sp^{d-1}$. Let also $1<u\leq v<\infty$, $\mu \in A_u$ and $\lambda \in A_v$. Define the Bloom weight $\nu$ by $\nu^{1+\alpha/d} := \mu^{1/u}\lambda^{-1/v}$, where $\alpha := d(\frac{1}{u}-\frac{1}{v})$.

    If $b\in \VMO^\alpha_\nu(\R^d)$, then the commutator $[b,T_\Omega]\colon L^u(\mu)\to L^v(\lambda)$ is compact.
\end{lemma}
\begin{proof}
    By Lemma \ref{lemma:smoothomegagivesczo}, $T_\Omega$ is a Calder\'{o}n-Zygmund operator in the sense of \cite[Definition 1.2]{HOS2023}. Thus by \cite[Theorem 2.4]{HOS2023}, the proof is done. (One does not need the non-degeneracy assumption of \cite{HOS2023} when one applies this ``direction'' of the theorem.)
\end{proof}

The following Theorem combines Lemma \ref{lemma:ifsmoothusehosandvmoforcompactness} with the Bloom-type boundedness of Corollary \ref{corollary:bloomBddforLqkernels}. The point in applying Lemma \ref{lemma:ifsmoothusehosandvmoforcompactness} is that $\Lip(\Sp^{d-1})$ is dense in $L^{q}(\Sp^{d-1})$ (for a similar proof strategy using the density of Lipschitz functions, see \cite{HLTY2023}).

\begin{theorem}\label{theorem:compactnessofcommutatorsbloomsetting}
    Let $1<q<\infty$. Suppose that $\Omega \in L^{q}(\Sp^{d-1})$ has zero average over $\Sp^{d-1}$. Let also $1<u\leq v<q$, $\mu \in A_u \cap RH_{(\frac{q}{u})'}$ and $\lambda \in A_v \cap RH_{(\frac{q}{v})'}$. Define the Bloom weight $\nu$ by $\nu^{1+\alpha/d} := \mu^{1/u}\lambda^{-1/v}$, where $\alpha := d(\frac{1}{u}-\frac{1}{v})$.

    If $b\in \VMO^\alpha_\nu(\R^d)$, then the commutator $[b,T_\Omega]\colon L^u(\mu)\to L^v(\lambda)$ is compact.
\end{theorem}
\begin{proof}
Suppose that $\eps>0$, and let $\Omega_\eps\in \Lip(\Sp^{d-1})$ such that  $\|\Omega_\eps-\Omega\|_{L^{q}(\Sp^{d-1})}<\eps$. Then the sequence 
\[
    \Omega_{\eps,0}\coloneqq\Omega_\eps-\fint_{\Sp^{d-1}}\Omega_\eps
\]
clearly satisfies $\Omega_{\eps,0}\in \Lip(\Sp^{d-1})$ and $\int_{\Sp^{d-1}}\Omega_{\eps,0}=0$. Furthermore, we have
\begin{align*}
    \|\Omega-\Omega_{\eps,0}\|_{L^{q}(\Sp^{d-1})}&\leq \|\Omega-\Omega_{\eps}\|_{L^{q}(\Sp^{d-1})}+\|\fint_{\Sp^{d-1}}\Omega_{\eps}\|_{L^{q}(\Sp^{d-1})}\\&= \|\Omega-\Omega_{\eps}\|_{L^{q}(\Sp^{d-1})}+|\Sp^{d-1}|^{\frac{1}{q}-1}\,|\int_{\Sp^{d-1}}\Omega_{\eps}-\Omega|
    \\&\leq \|\Omega-\Omega_{\eps}\|_{L^{q}(\Sp^{d-1})}+|\Sp^{d-1}|^{\frac{1}{q}-1+\frac{1}{q'}}\,\|\Omega-\Omega_{\eps}\|_{L^{q}(\Sp^{d-1})}
    \\&= 2\|\Omega-\Omega_{\eps}\|_{L^{q}(\Sp^{d-1})}
    \leq 2\eps.
\end{align*}
We decompose the commutator
\[
    [b,T_\Omega]=[b,T_{\Omega-\Omega_{\eps,0}+\Omega_{\eps,0}}]=[b,T_{\Omega-\Omega_{\eps,0}}]+[b,T_{\Omega_{\eps,0}}],
\]
and it suffices to show that $[b,T_{\Omega_{\eps,0}}]$ is compact and $[b,T_{\Omega-\Omega_{\eps,0}}]$ is arbitrarily small in the $L^u(\mu)\to L^v(\lambda)$ norm. By Lemma \ref{lemma:ifsmoothusehosandvmoforcompactness}, we only need to show the latter. By Corollary \ref{corollary:bloomBddforLqkernels} we have that
\begin{align*}
    \|[b,T_{\Omega-\Omega_{\eps,0}}]\|_{L^u(\mu)\to L^v(\lambda) }&\lesssim_{d,q,u,v,\mu,\lambda} \|\Omega-\Omega_{\eps,0}\|_{L^{q}(\Sp^{d-1})}\|b\|_{\mathrm{BMO}_\nu^{\alpha}}\lesssim \eps\,\|b\|_{\mathrm{BMO}_\nu^{\alpha}},
\end{align*}
which is the desired control. Note that $\|b\|_{\BMO_\nu^{\alpha}}<\infty$ by Lemma \ref{lemma:vmoinsidebmo}.
\end{proof}

\section{Necessity for Bloom-type compactness}\label{sec:necessityforcompactness}

By combining our result on the compactness of a commutator $[b,T_\Omega]$ with earlier research \cite{HOS2023}, we get a \emph{characterisation} of the compactness. We state the result next but a few words are in place after that to explain its proof.

\begin{theorem}\label{thm:compactnesscharac}
    Let $1<q<\infty$. Suppose that $\Omega \in L^q(\Sp^{d-1}) \setminus \{0\}$ has zero average over $\Sp^{d-1}$. Let also $1<u\leq v<q$, $\mu \in A_u \cap RH_{(\frac{q}{u})'}$ and $\lambda \in A_v \cap RH_{(\frac{q}{v})'}$. Define the Bloom weight $\nu$ by $\nu^{1+\alpha/d} := \mu^{1/u}\lambda^{-1/v}$, where $\alpha := d(\frac{1}{u}-\frac{1}{v})$. Let $b\in L^1_{\loc}(\R^d)$.

    Then the commutator $[b,T_\Omega]\colon L^u(\mu)\to L^v(\lambda)$ is compact if and only if $b\in \VMO^\alpha_\nu(\R^d)$.
\end{theorem}
\begin{proof}
    One direction of the equivalence is a consequence of Theorem \ref{theorem:compactnessofcommutatorsbloomsetting}. The other direction essentially follows from the proofs of \cite[Proposition 5.7]{HOS2023} and \cite[Theorem 2.4.1]{Hyt2021}: see the remarks below.
\end{proof}

In \cite{HOS2023}, the authors do not state their results for $T$ associated to a rough homogeneous kernel $K_\Omega$. Rather, they work with kernels that satisfy a certain smoothness condition. However, their result on the necessary condition for the compactness of the commutator \cite[Proposition 5.7]{HOS2023} is also valid for rough homogeneous kernels $K_\Omega$. This validity is what suffices to show that Theorem \ref{thm:compactnesscharac} is true. In fact, the proof of \cite[Proposition 5.7]{HOS2023} works as such for the rough homogeneous kernels. We explain this next.

In the context of \cite[Proposition 5.7]{HOS2023}, non-degeneracy of $T=T_\Omega$ is interpreted to mean that $\Omega$ is non-zero in a set of positive measure, that is, $\Omega \in L^q(\Sp^{d-1}) \setminus \{0\}$. 

Arguably, the least trivial part of why the proof of \cite{HOS2023} works for non-degenerate $T_\Omega$ is that a method called \emph{approximate weak factorisation} is valid for the non-degenerate kernels of the \emph{rough homogeneous} kind as well, as demonstrated in \cite[Lemma 2.3.3]{Hyt2021}. Therefore, if one follows the proof of \cite[Theorem 2.4.1]{Hyt2021}, then one gets the following lemma that resembles \cite[Proposition 4.2]{HOS2023}, as a consequence of the approximate weak factorisation:

\begin{lemma}\label{lemma:conseqeunceofawf}
    Suppose that $\Omega \in L^q(\Sp^{d-1}) \setminus \{0\}$ has zero average over $\Sp^{d-1}$ and $b\in L^1_{\loc}(\R^d)$. Let $Q$ be a fixed cube. Then there exists a cube $\wt{Q}$ such that $\mathrm{dist}(Q,\wt{Q}) \eqsim \ell(Q)=\ell(\wt{Q})$ and for any $\gamma$-major subsets $E\subset Q$ and $\wt{E} \subset \widetilde{Q}$ (that is, $|E|\geq \gamma|Q|$ and $|\wt{E}|\geq\gamma|\wt{Q}|$) we have
    \begin{equation}\label{eq:awfconclusion}
    \int_E|b-\ave{b}_E| \dif x \lesssim |\langle [b,T_\Omega]h_E, g_{\wt{E}}\rangle| + |\langle [b,T_\Omega]g_E, h_{\wt{E}}\rangle|,
    \end{equation}
    where the auxiliary functions satisfy 
    \[
    g_E=1_E, \qquad g_{\wt{E}} = 1_{\wt{E}}, \qquad h_E\lesssim 1_E, \qquad h_{\wt{E}} \lesssim 1_{\wt{E}}.
    \]
    All the above implicit constants depend at most on the kernel of $T_\Omega$, the dimension $d$ and $\gamma$.
\end{lemma}

Thus the black box (the approximate weak factorisation) that is used in the proof of \cite[Proposition 5.7]{HOS2023} is also valid for non-degenerate rough homogeneous kernels. In addition to this part, the other parts of the proof are easily seen to be independent of which type of kernel is in consideration.

\begin{remark}
Note: The conclusion of \cite[Proposition 4.2]{HOS2023} contains a typo that does not affect the main results of \cite{HOS2023}. The typo has been corrected above in \eqref{eq:awfconclusion}. Indeed, the commutator should be applied to the auxiliary functions that are supported in $E$ (as is the case in \eqref{eq:awfconclusion}), not to the ones supported in $\wt{E}$. In the paper \cite{HOS2023}, a trivial modification is required in the proof of \cite[Proposition 4.3]{HOS2023} but the typo has no other implications.
\end{remark}

\section{Matrix-weighted compactness}\label{sec:matrixweightedcompactness}
A matrix weight is a locally integrable function $W \colon \R^d \to \CC^{
n\times n}$ that is almost everywhere
positive definite -valued. The space
$L^p(W)$ consists of all measurable $\vec f \colon \R^d \to \CC^n$ such that $W^\frac{1}{p} \vec f \in L^p(\R^d,\CC^n)$,
and $\|
\vec f\|_{L^p(W)}
\coloneqq \|W^\frac{1}{p} \vec f\|_{L^p(\R^d,\CC^n)}$. For a matrix weight $W \colon \R^d \to \CC^{
n\times n}$ and $1<p<\infty$, we
use the definition of $A_p$ that appeared for the first time in \cite{Rou2003}, i.e.,
\[
    [W]_{A_p}\coloneqq\sup\limits_{Q}\fint_Q\left(\fint_Q\Big|W^{\frac{1}{p}}(x)W^{-\frac{1}{p}}(y)\Big|^{p'}_{op}\dif y\right)^\frac{p}{p'}\dif x
\]
and say that $W\in A_p$ if $[W]_{A_p}<\infty$. We will use the matrix-weighted maximal function $M_W$ of Goldberg and Christ \cite{CG2001,Gol2003}, which is defined by 
\begin{equation}\label{eq:matrixweightedmaxfunc}
    M_W\vec f(x)\coloneqq \sup_{Q\ni x}\fint_{Q}|W^\frac{1}{p}(x)W^{-\frac{1}{p}}(y)\vec f(y)|\dif y.
\end{equation}
The main property of $M_W$ that will be used is the boundedness of $M_W$ on $L^p(\R^d)$ when $W\in A_p$. See \cite[Theorem 3.2]{Gol2003} and \cite[Theorem 1.3]{IM2019} for more details on this fact.

In this section we will strive to prove a matrix-weighted compactness result for $[b,T_\Omega]$. The main ingredient of the proof of this result is the following matrix-weighted Kolmogorov-Riesz compactness theorem due to \cite[Corollary 3.2]{LYZ2023}.
\begin{theorem}\label{thm:kolmogorovriesz}
    Let $1<p<\infty$ and let $W$ be a matrix weight. A subset $\mathscr F$ of $L^p(W)$ is totally bounded if the following conditions hold:
    \begin{enumerate}
        \item[(a)]$\mathscr F$ is bounded, that is, \[\sup_{\vec f\in\mathscr F}\|\vec f\|_{L^p(W)}<\infty;\]
        \item[(b)] $\mathscr F$ uniformly vanishes at infinity, that is,
        \[
            \lim_{R\to\infty}\sup_{\vec f\in\mathscr F}\|\vec f \,\mathds{1}_{B(0,R)^\complement}\|_{L^p(W)}=0;
        \]
        \item[(c)] $\mathscr F$ is equicontinuous, that is,
        \[
            \lim_{r\to0}\sup_{\vec f\in\mathscr F}\sup_{z\in B(0,r)}\|\tau_z\vec f-\vec f\|_{L^p(W)}=0,
        \]
    where $\tau_z$ is the translation operator defined by
    \[
        \tau_z\vec f(x)\coloneqq \vec f(x+z).
    \]
    \end{enumerate}
\end{theorem}

Now we will prove that the commutator $[b,T_\Omega]$ of a rough singular integral $T_\Omega$ with $b\in \mathrm{VMO}(\R^d)$ is $L^p(W)\to L^p(W)$ compact. The proof has similar elements with the scalar-valued proofs in \cite{GHWY2022,GWY2021}.

\begin{theorem}\label{thm:matrixweightcompactness}
    Let $W$ be a matrix weight and $b\in \operatorname{VMO}(\R^d)$. Let also $1<q<\infty$, $q'<p<\infty$, $W\in A_{\frac{p}{q'}}$ and $\Omega\in L^{q+\eps}(\Sp^{d-1})$, $\eps>0$, has zero average over $\Sp^{d-1}$. Then the commutator $[b,T_\Omega]\colon L^p(W)\to L^p(W)$ is compact.
\end{theorem}
\begin{proof}
    Due to \cite[Corollary 6.3]{Lau2024} and Lemma \ref{lemma:lorentzorlicztolebesgue}, we may assume $\Omega\in\operatorname{Lip}(\Sp^{d-1})$. Furthermore, a  result proved in \cite{Uch1978} says that $\VMO(\R^d)$ is the closure of $C^\infty_c(\R^d)$ under the $\BMO(\R^d)$ norm, so we may also assume that $b\in C^\infty_c(\R^d)$.
    
    Let $\phi_\delta(x)\coloneqq\phi(x/\delta)$, where $0\leq \phi\leq 1$ 
    is a smooth function supported on $B(0,1)$ and equal to one on $B(0,1/2)$. Then we define $T_{\Omega,\delta}$ to be the operator with kernel $K_{\Omega,\delta}(x,y)\coloneqq (1-\phi_\delta(x-y))K_\Omega(x,y)$. By the mean value theorem and kernel estimates, we have  
    \begin{align*}
        |W^\frac{1}{p}(x)([b,T_\Omega]\vec f(x)-&[b,T_{\Omega,\delta}]\vec f(x))|\\&=\Big|\text{p.v.}\int_{\R^d}(b(x)-b(y))\phi_\delta(x-y)K_{\Omega}(x,y)W^\frac{1}{p}(x)\vec f(y)\dif y\Big|\\&\lesssim_{b,\Omega} \int_{|x-y|\leq \delta}\frac{1}{|x-y|^{d-1}}\,|W^\frac{1}{p}(x)\vec f(y)|\dif y\\&= \int_{|x-y|\leq \delta}\frac{1}{|x-y|^{d-1}}\,|W^\frac{1}{p}(x)W^{-\frac{1}{p}}(y)W^\frac{1}{p}(y)\vec f(y)|\dif y.
    \end{align*}
    Then we decompose the last integral and estimate
    \begin{align*}
        \sum_{j=0}^\infty\int_{2^{-j-1}\delta<|x-y|\leq 2^{-j}\delta}&\,\frac{1}{|x-y|^{d-1}}\,|W^\frac{1}{p}(x)W^{-\frac{1}{p}}(y)W^\frac{1}{p}(y)\vec f(y)|\dif y\\ \lesssim &\,\delta\sum_{j=0}^\infty2^{-j} \fint_{|x-y|\leq 2^{-j}\delta}|W^\frac{1}{p}(x)W^{-\frac{1}{p}}(y)W^\frac{1}{p}(y)\vec f(y)|\dif y
         \\\lesssim &\,\delta M_W(W^\frac{1}{p}\vec f)(x),
    \end{align*}
    where $M_W$ is the matrix-weighted maximal function \eqref{eq:matrixweightedmaxfunc}.
    Due to the boundedness of the matrix-weighted maximal function we get
    \[
        \|[b,T_\Omega]\vec f-[b,T_{\Omega,\delta}]\vec f\|_{L^p(W)}\lesssim_{b,\Omega}\delta\, \|M_W(W^\frac{1}{p}\vec f)\|_{L^p(\R^d)}\lesssim_{W} \delta\,\|\vec f\|_{L^p(W)}
    \]
    Since $\delta$ can be arbitrarily small, it suffices to show that $[b,T_{\Omega,\delta}]\colon L^p(W)\to L^p(W)$ is compact. We note that the truncated kernel satisfies 
    \begin{equation}\label{eq:smoothtruncatedkernelest}
        |K_{\Omega,\delta}(x,y)-K_{\Omega,\delta}(x',y)|\lesssim \frac{|x-x'|}{|x-y|^{d+1}},\quad 2|x-x'|\leq |x-y|.
    \end{equation}
    
    By \cite[Corollary 6.3]{Lau2024} we have that $[b,T_{\Omega}]$ is $L^p(W)\to L^p(W)$ bounded, and hence
    \begin{align*}
        \|[b,T_{\Omega,\delta}]\vec f\|_{L^p(W)}&\leq \|[b,T_{\Omega,\delta}]\vec f-[b,T_{\Omega}]\vec f\|_{L^p(W)}+\|[b,T_{\Omega}]\vec f\|_{L^p(W)}\\&\lesssim_{b,W,\Omega}\|\vec f\|_{L^p(W)}<\infty,
    \end{align*}
    which takes care of the first part of Theorem \ref{thm:kolmogorovriesz}.

    Considering the second part, we suppose that $Q$ is a cube centered at the origin such that $\supp b\subset Q$. Then for $x$ sufficiently far away from the origin we have
    \begin{align*}
        |W^\frac{1}{p}(x)[b,T_{\Omega,\delta}]\vec f(x)|^p&=\Big|W^\frac{1}{p}(x)\int_{\R^d}b(y)K_{\Omega,\delta}(x,y)\vec f(y)\dif y\Big|^p\\&\lesssim_{b,\Omega}|W^\frac{1}{p} (x) |_{op}^p\frac{1}{|x|^{dp}}\left(\int_{Q}|\vec f(y)|\dif y\right)^p\\&\leq|W(x)|_{op}\frac{1}{|x|^{dp}}\left(\int_{Q}|\vec f(y)|\dif y\right)^p,
    \end{align*}
    where the last inequality is true due to the Cordes inequality.
    Moreover, an application H\"older's inequality and another application of the Cordes inequality yields
    \[
        \int_{Q}|\vec f(y)|\dif y\leq \int_{Q}|W^{-\frac{1}{p}}(y)|_{op}|W^\frac{1}{p}(y)\vec f(y)|\dif y\leq \left(\int_Q|W^{-\frac{p'}{p}}(y)|_{op}\dif y\right)^\frac{1}{p'}\|\vec f\|_{L^p(W)}.
    \]
    Thus for sufficiently large $N>0$ we have
    \begin{align*}
        &\left(\int_{(B(0,2^N))^\complement}|W^\frac{1}{p}(x)[b,T_{\Omega,\delta}]\vec f(x)|^p\dif x\right)^\frac{1}{p}\\&\qquad\qquad\qquad\lesssim_{b,\Omega}\|\vec f\|_{L^p(W)}\left(\int_Q|W^{-\frac{p'}{p}}(y)|_{op}\dif y\right)^\frac{1}{p'}\left(\int_{(B(0,2^N))^\complement}\frac{|W(x)|_{op}}{|x|^{dp}}\dif x\right)^\frac{1}{p}.
    \end{align*}
    We estimate the last term of the product on the right-hand side as follows
    \begin{align*}
        \int_{(B(0,2^N))^\complement}
        \frac{|W(x)|_{op}}{|x|^{dp}}\dif x
        &\leq\sum_{k=\lfloor N\rfloor}^\infty \int_{B(0,2^{k+1})\setminus B(0,2^k)}\frac{|W(x)|_{op}}{|x|^{dp}}\dif x\\&\leq \sum_{k=\lfloor N\rfloor}^\infty 2^{-dkp}\int_{B(0,2^{k+1})\setminus B(0,2^k)}|W(x)|_{op}\dif x.
    \end{align*}
    We note that for any $w\in A_p$ we have the doubling property
    \[
        w(B(0,r))\lesssim \frac{r^{dp}}{s^{dp}}[w]_{A_{p}}w(B(0,s)), \quad r\geq s>0,
    \]
    and there exists a $\gamma>0$ such that $w\in A_{p-\gamma}$. Applying these facts with $w\coloneqq |W|_{op}\in A_p$ yields  
    \begin{align*}
        \int_{B(0,2^{k+1})\setminus B(0,2^k)}w(x)\dif x&\lesssim [w]_{A_p}w(B(0,2^k))\lesssim[w]_{A_p}2^{dkp}2^{-dk\gamma}[w]_{A_{p-\gamma}}w(B(0,1)),
    \end{align*}
    and hence
    \begin{align*}
        \int_{(B(0,2^N))^\complement}
        \frac{w(x)}{|x|^{dp}}\dif x&\lesssim[w]_{A_p}[w]_{A_{p-\gamma}}w(B(0,1)) \sum_{k=\lfloor N\rfloor}^\infty 2^{-dk\gamma}\\&\eqsim 2^{-dN\gamma}[w]_{A_p}[w]_{A_{p-\gamma}}w(B(0,1)).
    \end{align*}
    Therefore, with
    \[C(W)\coloneqq [|W|_{op}]_{A_{p}}[|W|_{op}]_{A_{p-\gamma}}\int_{B(0,1)}|W(y)|_{op}\dif y\left(\int_Q|W^{-\frac{p'}{p}}(y)|_{op}\dif y\right)^\frac{p}{p'}<\infty\] we have
    \[
        \left(\int_{(B(0,2^N))^\complement}|W^\frac{1}{p}(x)[b,T_{\Omega,\delta}]\vec f(x)|^p\dif x\right)^\frac{1}{p}\lesssim_{b,\Omega} 2^{-\frac{dN\gamma}{p}}C(W)^\frac{1}{p}\|\vec f\|_{L^p(W)}\xrightarrow{N\to\infty} 0,
    \]
    which concludes the proof of the second item.

    All that is left is the proof of the last item in Theorem \ref{thm:kolmogorovriesz}. To this end, we take $z\in\R^d$ with $|z|\leq \delta/8$ and write
    \begin{align*}
        [b,T_{\Omega,\delta}]\vec f(x+z)-
        [b,T_{\Omega,\delta}]\vec f(x)= I(x,z)+II(x,z),
    \end{align*}
    where 
    \[I(x,z)\coloneqq \int_{\R^d}(b(x+z)-b(y))(K_{\Omega,\delta}(x+z,y)-K_{\Omega,\delta}(x,y))\vec f(y)\dif y\]
    and 
    \[
        II(x,z)\coloneqq (b(x+z)-b(x))\int_{\R^d}K_{\Omega,\delta}(x,y)\vec f(y)\dif y.
    \]
    Due to the fact that the terms $K_{\Omega,\delta}(x+z,y)$ and $K_{\Omega,\delta}(x,y)$ vanish for $|x-y|\leq \delta/4$ and \eqref{eq:smoothtruncatedkernelest} we have
    \begin{align*}
        |W^\frac{1}{p}(x)I(x,z)|&\leq 2\|b\|_{\infty}\int_{|x-y|\geq \frac{\delta}{4}}|K_{\Omega,\delta}(x+z,y)-K_{\Omega,\delta}(x,y)|\,|W^\frac{1}{p}(x)\vec f(y)|\dif y
        \\&\lesssim_{b}\int_{|x-y|\geq \frac{\delta}{4}}\frac{|z|}{|x-y|^{d+1}}|W^\frac{1}{p}(x)\vec f(y)|\dif y \\&=|z|\sum_{j=-2}^\infty\int_{2^{j+1}\delta>|x-y|\geq 2^j\delta}\frac{1}{|x-y|^{d+1}}|W^\frac{1}{p}(x)\vec f(y)|\dif y
        \\&\lesssim\frac{|z|}{\delta}\sum_{j=-2}^\infty2^{-j}\fint_{|x-y|\leq 2^{j+1}\delta}|W^\frac{1}{p}(x)W^{-\frac{1}{p}}(y)W^\frac{1}{p}(y)\vec f(y)|\dif y
        \\&\lesssim\frac{|z|}{\delta}M_{W}\left(W^\frac{1}{p}\vec f\,\right)(x).
    \end{align*}
    This yields 
    \begin{equation}\label{eq:Ibound}
        \|I(\cdot,z)\|_{L^p(W)}\lesssim_{b,\delta}|z|\left\|M_W\left(W^\frac{1}{p}\vec f\,\right)\right\|_{L^p(\R^d)}\lesssim_{W}|z|\,\|\vec f\|_{L^p(W)}.
    \end{equation}
    For the second term we apply the mean value theorem to get $|b(x+z)-b(x)|\lesssim_b|z|$,
    and hence
    \begin{align*}
        |II(x,z)|&\lesssim_{b} |z|\left(|\int_{\delta/2\leq |x-y|\leq\delta}K_{\Omega,\delta}(x,y)\vec f(y)\dif y|+|\int_{|x-y|>\delta}K_{\Omega}(x,y)\vec f(y)\dif y|\right)\\&\eqqcolon |z|\,\Big(|II_1(x)|+|II_2(x)|\Big).
    \end{align*}
    The first term can be estimated with the matrix-weighted maximal function in a similar manner as before
    \[
        |W^\frac{1}{p}(x)II_1(x)|\lesssim \fint_{|x-y|\leq \delta}|W^\frac{1}{p}(x)W^{-\frac{1}{p}}(y)W^\frac{1}{p}(y)\vec f(y)|\dif y\lesssim M_W(W^\frac{1}{p}\vec f)(x),
    \]
    and the second term is a truncated singular integral, which is bounded on $L^p(W)$ (see for instance \cite[Corollary 1.6]{DTL2020} together with \cite[Theorem 1.4]{BC2023}). Thus we get
    \begin{equation}\label{eq:IIbound}
        \|II(\cdot,z)\|_{L^p(W)}\lesssim_{b,W}|z|\|\vec f\|_{L^p(W)}.
    \end{equation}
    Combining the estimates \eqref{eq:Ibound} and \eqref{eq:IIbound} shows that
    \[
        \|\tau_z[b,T_{\Omega,\delta}]\vec f-[b,T_{\Omega,\delta}]\vec f\|_{L^p(W)}\to 0
    \]
    as $|z|\to0$.

    We have now checked that for a bounded set $B\subset L^p(W)$, the image $[b,T_{\Omega,\delta}]B$ satisfies the conditions of Theorem \ref{thm:kolmogorovriesz}. Thus by Theorem \ref{thm:kolmogorovriesz} it follows that $[b,T_{\Omega,\delta}]B$ is totally bounded in $L^p(W)$, and this implies that $[b,T_{\Omega,\delta}]$ is a compact operator. 
\end{proof}
\begin{remark}
    A result of Bownik \cite[Proposition 5.3]{Bow2001} says that generally matrix-weights do not enjoy the self-improvement property of $A_p$ classes. Therefore, we are not able to get the result of Theorem \ref{thm:matrixweightcompactness} for $\Omega\in L^q(\Sp^{d-1})$. 
\end{remark}

\begin{remark}
    To the authors' best knowledge this is the first matrix-weighted compactness result for a commutator of a singular integral. With relatively small changes to the proof, one can prove matrix-weighted compactness for a Calder\'on-Zygmund operator with a standard kernel like those studied for example in \cite{HOS2023}.
\end{remark}
\begin{remark}
    It is known in the scalar-valued setting that $b\in \operatorname{VMO}$ is necessary for the compactness of $[b,T_\Omega]$ when $\Omega\neq 0$ (see Theorem \ref{thm:compactnesscharac}). The scalar-valued case is a special case of the matrix-valued case and thus in this sense $b\in \operatorname{VMO}$ is necessary in the more general matrix-valued setting.
\end{remark}

\appendix
\section{Lipschitz functions induce kernels that satisfy size and smoothness estimates of standard kernels}\label{appendixa}
For easy reference for the interested reader, we recall the missing pieces of the proof of Lemma \ref{lemma:smoothomegagivesczo} that says that Lipschitz-functions $\Omega$ induce standard kernels $K_\Omega$ of Calder\'on-Zygmund operators, when we define
\[
K_\Omega(x,y):=\Omega\left(\frac{x-y}{|x-y|}\right)\frac{1}{|x-y|^d}
\]
for all $x,y\in \R^d$ such that $x\neq y$.

Because we extend $\Omega$ from $\Sp^{d-1}$ to $\R^d$ keeping it constant on each ray starting from the origin, we see what the Lipschitz condition of $\Omega$ transforms into in the following lemma. 

\begin{lemma}\label{lemma:howlipschitzspreadstowholespace}
    Suppose that $\Omega \in \Lip(\Sp^{d-1})$ with Lipschitz constant $M$. Let us extend $\Omega$ to the domain $\R^d\setminus\{0\}$ by setting $\Omega(x):=\Omega(x/|x|)$. Then 
    \[
    |\Omega(x)-\Omega(y)| \leq 2M \min\left\{\frac{|x-y|}{|x|}, \frac{|x-y|}{|y|}\right\}, \quad \forall x,y\in \R^d\setminus \{0\}. 
    \]
\end{lemma}
\begin{proof}
    Let $x,y\in \R^d\setminus \{0\}$. Then 
    \[
    |\Omega(x)-\Omega(y)| = |\Omega(x/|x|)-\Omega(y/|y|)| \leq M\left|\frac{x}{|x|}-\frac{y}{|y|}\right|.
    \]
    Note that 
    \[
    \left|\frac{x}{|x|}-\frac{y}{|y|}\right| = \frac{||y|x-|x|y|}{|x||y|}=\begin{cases}
        &\frac{||y|(x-y)+(|y|-|x|)y|}{|x||y|}, \\
        &\frac{|(|y|-|x|)x+|x|(x-y)|}{|x||y|}.
    \end{cases}
    \]
    Applying the triangle inequality to both expressions we get 
    \[
    \left|\frac{x}{|x|}-\frac{y}{|y|}\right|\leq \begin{cases}
        &2\frac{|x-y|}{|x|}, \\
        &2\frac{|x-y|}{|y|}.
    \end{cases}
    \]
\end{proof}

We are now to ready to prove the size and smoothness conditions for $K_\Omega$. This provides the missing argument in the proof of Lemma \ref{lemma:smoothomegagivesczo}.

\begin{lemma}
Suppose that $\Omega \in \Lip(\Sp^{d-1})$. Then 
\[
|K_\Omega(x,y)| \leq \frac{\|\Omega\|_{L^\infty}}{|x-y|^d} \quad (\forall x,y\in \R^d , x\neq y).
\]
Furthermore,
\[
|K_\Omega(x',y) - K_\Omega(x,y)| + |K_\Omega(y,x') - K_\Omega(y,x)| \leq \omega\left(\frac{|x-x'|}{|x-y|}\right)\frac{1}{|x-y|^d}
\]
whenever $|x-x'| \leq \frac{1}{2}|x-y|$, $x\neq y$. We may choose $\omega$ so that it satisfies $\omega(t)\eqsim (t)$ for all $t\in[0,1]$.
\end{lemma}
\begin{proof}
    Denote $K_\Omega(x,y):=\Omega((x-y)/|x-y|)|x-y|^{-d}$ for all $x,y\in \R^d$ such that $x\neq y$. Because $\Omega \in \Lip(\Sp^{d-1})$, there exists a constant $M>0$ such that $|\Omega(x)-\Omega(y)|\leq M|x-y|$ for all $x,y\in \Sp^{d-1}$. Note that the assumed Lipschitz continuity implies that 
    \[
    |\Omega(x)-\Omega((1,0,\ldots,0))| \leq M|x-(1,0,\ldots,0)| \leq 2M \quad (\forall x\in \Sp^{d-1}). 
    \]
    Thus $\Omega(\Sp^{d-1})$ is bounded and 
    \[
    |K_\Omega(x,y)| = \frac{|\Omega((x-y)/|x-y|)|}{|x-y|^d} \leq \frac{\|\Omega\|_{L^\infty}}{|x-y|^d} \quad (\forall x,y\in \R^d , x\neq y).
    \]
    
    For the purpose of clarity, let us extend $\Omega$ to the domain $\R^d\setminus\{0\}$ by setting $\Omega(x):=\Omega(x/|x|)$. Whenever $x,x',y\in \R^d$ are such that $|x-x'| \leq \frac{1}{2}|x-y|$ (and $x\neq y$), we have that
    \begin{align*}
        &|K_\Omega(x',y) - K_\Omega(x,y)| + |K_\Omega(y,x') - K_\Omega(y,x)| \\
        &= \left|\frac{\Omega(x'-y)}{|x'-y|^d} - \frac{\Omega(x-y)}{|x-y|^d}\right| + \left|\frac{\Omega(y-x')}{|y-x'|^d} - \frac{\Omega(y-x)}{|y-x|^d}\right| \\
        &=\left|\frac{|x-y|^d\Omega(x'-y)-|x'-y|^d\Omega(x-y)}{|x'-y|^d|x-y|^d}\right| \!\! + \!\! \left|\frac{|y-x|^d\Omega(y-x')-|y-x'|^d\Omega(y-x)}{|y-x'|^d|y-x|^d}\right| \\
        &=: \mathrm{I} + \mathrm{II}. 
    \end{align*}
    Here
    \begin{align*}
        \mathrm{I} \leq 2^d\left|\frac{|x-y|^d\Omega(x'-y)-|x'-y|^d\Omega(x-y)}{|x-y|^{2d}}\right|,
    \end{align*}
    because 
    \[
    |x-y| \leq |x-x'| + |x'-y| \leq \frac{1}{2}|x-y| + |x'-y|
    \]
    and hence $|x-y| \leq 2|x'-y|$. We write further that
    \begin{align*}
    \mathrm{I} &\leq 2^d\left|\frac{(|x-y|^d-|x'-y|^d)\Omega(x-y)}{|x-y|^{2d}}\right| + 2^d\left|\frac{|x-y|^d(\Omega(x'-y)-\Omega(x-y))}{|x-y|^{2d}}\right|  \\
    & =: \mathrm{J} + \mathrm{JJ}.
    \end{align*}
    Note that 
    \begin{align*}
    \mathrm{J} &= 2^d |\Omega(x-y)||x-y|^{-2d}||x-y|^d-|x'-y|^d| \\
    &\leq 2^d \|\Omega\|_{L^\infty} |x-y|^{-2d}||x-y|^d-|x'-y|^d|.
    \end{align*}
    By the mean value theorem applied to the function $z\mapsto z^d$, we get 
    \[
    ||x-y|^d-|x'-y|^d| \leq d\max\{|x-y|,|x'-y|\}^{d-1}|x-x'| \leq 2^{d-1}d |x-y|^{d-1}|x-x'|,
    \]
    where in the last step we used that 
    \[
    |x'-y| \leq |x-x'|+|x-y| \leq 2|x-y|.
    \]
    Thus
    \[
    \mathrm{J} \leq 2^{2d-1}d\|\Omega\|_{L^\infty} |x-y|^{-d} \frac{|x-x'|}{|x-y|}.
    \]
    On the other hand, note that by Lemma \ref{lemma:howlipschitzspreadstowholespace} we have that
    \begin{align*}
        \mathrm{JJ} &= 2^d |x-y|^{-d} |\Omega(x'-y)-\Omega(x-y)| \\
        &\leq 2^{d+1} M |x-y|^{-d} \min\left\{\frac{|x-x'|}{|x'-y|},\frac{|x-x'|}{|x-y|}\right\} \\
        &\leq 2^{d+2} M |x-y|^{-d} \frac{|x-x'|}{|x-y|}.
    \end{align*}
    Thus we get that 
    \[
    \mathrm{I} \leq \omega_{\mathrm{I}}\left(\frac{|x-x'|}{|x-y|}\right)|x-y|^{-d},
    \]
    where the modulus of continuity $\omega_{\mathrm{I}}$ satisfies 
    \[
    \omega_{\mathrm{I}}(t) \eqsim t.
    \]
    In particular $\omega_{\mathrm{I}}$ is increasing, subadditive and satisfies $\omega_{\mathrm{I}}(0)=0$ as well as the Dini condition
    \[
    \int_0^1\omega_{\mathrm{I}}(t)\frac{\dif t}{t}<\infty.
    \]
    The estimate for II is similar. In particular, one estimates  
    \begin{align*}
    \mathrm{II} &\leq 2^d\left|\frac{(|x-y|^d-|x'-y|^d)\Omega(y-x)}{|x-y|^{2d}}\right| + 2^d\left|\frac{|x-y|^d(\Omega(y-x')-\Omega(y-x))}{|x-y|^{2d}}\right| \\
    &=: \mathrm{K} + \mathrm{KK}
    \end{align*}
    and then continues with the same strategy as for I and ends up with the same modulus of continuity. Combining the estimates for I and II, we get 
    \[
    |K_\Omega(x',y) - K_\Omega(x,y)| + |K_\Omega(y,x') - K_\Omega(y,x)| \leq \omega\left(\frac{|x-x'|}{|x-y|}\right)|x-y|^{-d},
    \]
    where $\omega := 2\omega_{\mathrm{I}}$ is increasing, subadditive and satisfies $\omega(0)=0$ as well as the Dini condition.
\end{proof}

\subsection*{Acknowledgements}
{The authors thank Tuomas Hyt\"onen for helpful comments on the manuscript. 

J.S.\ was supported by the Research Council of Finland through project 358180 (to Timo H\"{a}nninen). A.L. is supported by the Research Council of Finland through the Finnish Centre of Excellence in Randomness and Structures (grant No. 364208 to Tuomas Hyt\"onen).
}

\bibliographystyle{plain}
\bibliography{mainreferences}

\end{document}